\documentclass[reqno,12pt]{amsart}

\usepackage{enumerate}
\usepackage[margin=1in]{geometry}
\usepackage{ifpdf}
\usepackage{amsmath}
\usepackage{amsfonts}
\usepackage{amssymb}
\usepackage{amsthm}
\usepackage[hyperfootnotes=false]{hyperref}
\usepackage{setspace}
\usepackage{amsrefs}
\usepackage{nicefrac}
\usepackage{graphicx}

\newcommand{\ignore}[1]{}

\newcommand{\linnprod}[2]{\langle #1 , #2 \rangle}

\newcommand{\abs}[1]{\left\lvert {#1} \right\rvert}
\newcommand{\sabs}[1]{\lvert {#1} \rvert}
\newcommand{\norm}[1]{\left\lVert {#1} \right\rVert}
\newcommand{\snorm}[1]{\lVert {#1} \rVert}

\newcommand{\C}{{\mathbb{C}}}
\newcommand{\R}{{\mathbb{R}}}

\newcommand{\N}{{\mathbb{N}}}

\newcommand{\bB}{{\mathbb{B}}}

\newcommand{\bP}{{\mathbb{P}}}

\newcommand{\sA}{{\mathcal{A}}}

\newcommand{\sF}{{\mathcal{F}}}

\newcommand{\sN}{{\mathcal{N}}}
\newcommand{\sO}{{\mathcal{O}}}

\newcommand{\sI}{{\mathcal{I}}}

\newtheorem{thm}{Theorem}[section]

\newtheorem{prop}[thm]{Proposition}

\newtheorem{cor}[thm]{Corollary}

\newtheorem{lemma}[thm]{Lemma}

\theoremstyle{definition}
\newtheorem{defn}[thm]{Definition}
\newtheorem{example}[thm]{Example}

\theoremstyle{remark}
\newtheorem{remark}[thm]{Remark}

\newcommand{\afspan}[1]{\operatorname{affine-span} #1}
\newcommand{\ballsig}[2]{{\bB}_{#2}^{#1+#2}}
\newcommand{\ballsigfull}[3]{{\bB}_{#2}^{#3}}

\author{Dusty Grundmeier}
\address{Department of Mathematics, Ball State University, Muncie, IN 47306,
USA}
\curraddr{Department of Mathematics, Harvard University, Cambridge MA 02138,
USA}
\email{deg@math.harvard.edu}

\author{Ji\v{r}\'{\i} Lebl}
\thanks{The second author was in part supported by NSF grant DMS-1362337 and
Oklahoma State University's DIG and ASR grants.}
\address{Department of Mathematics, Oklahoma State University,
Stillwater, OK 74078, USA}
\email{lebl@math.okstate.edu}

\date{September 9, 2015}

\ifpdf
\hypersetup{
  pdftitle={Initial monomial invariants of holomorphic maps},
  pdfauthor={Dusty Grundmeier and Jiri Lebl},
  pdfsubject={Several Complex Variables, CR Geometry},
  pdfkeywords={Proper holomorphic maps, initial monomial ideal},
}
\fi

\title%
{Initial monomial invariants of holomorphic maps}

\begin{document}

%\doublespace

\begin{abstract}
We study a new biholomorphic invariant of holomorphic maps between domains
in different dimensions based on generic initial ideals.
We start with the standard generic monomial ideals to find invariants
for rational maps of spheres and hyperquadrics, giving a readily computable
invariant in this important case.
For example, the generic initial
monomials distinguish all four inequivalent rational proper
maps from the two to the three dimensional ball.
Next,
we associate
to each subspace $X \subset \sO(U)$
a generic initial monomial subspace, which is invariant under
biholomorphic transformations and multiplication by nonzero functions.
The generic initial monomial subspace is a biholomorphic
invariant for holomorphic maps if the target automorphism
is linear fractional as in the case of automorphisms of
spheres or hyperquadrics.
\end{abstract}

\maketitle

%\enlargethispage{\baselineskip}

%%%%%%%%%%%%%%%%%%%%%%%%%%%%%%%%%%%%%%%%%%%%%%%%%%%%%%%%%%%%%%%%%%%%%%%%%%%%

\section{Introduction} \label{section:intro}

Let $U \subset \C^n$ and $V \subset \C^m$ be domains.
Denote by $\sO(U,V)$
the set of holomorphic maps 
\begin{equation}
f \colon U \to V .
\end{equation}
Write $\sO(U) = \sO(U,\C)$ as usual.
A fundamental problem in several complex variables is to understand maps
in $\sO(U,V)$
up to automorphisms; that is, $f \colon U \to V$ and
$g \colon U \to V$ are equivalent if there exist biholomorphisms
$\tau \in \operatorname{Aut}(U)$ and
$\chi \in \operatorname{Aut}(V)$ such that
\begin{equation}
f \circ \tau = \chi \circ g .
\end{equation}
A particularly important case is when the maps are proper. A
map $f \colon U \to V$ is proper if for every compact $K \subset \subset V$, the
set $f^{-1}(K)$ is compact.
If $f$ extends continuously to the boundary, then the proper
map $f$ takes boundary to boundary.  The map $f$ restricted to the boundary gives a CR map, and therefore, we have a problem in CR geometry.
It is important to understand those situations where $V$
has a large automorphism group, and thus we focus most of our attention on the unit ball of a certain signature.
For a pair of
integers $(a,b)$, $a \geq 1$, the unit ball of signature $b$
is given by
\begin{equation}
\ballsig{a}{b} = \Bigl\{ z = (z_1,\ldots,z_{a+b}) \in \C^{a+b} :
- \sum_{j=1}^b \abs{z_j}^2 + \sum_{j=1+b}^{a+b} \abs{z_j}^2 < 1 \Bigr\} .
\end{equation}
The ball $\ballsigfull{n}{0}{n}$ with signature 0 is
the standard unit ball $\bB^n$.
The boundary of $\ballsig{a}{b}$ is the hyperquadric
\begin{equation}
Q(a,b) = \Bigl\{ z \in \C^{a+b} :
- \sum_{j=1}^b \abs{z_j}^2 + \sum_{j=1+b}^{a+b} \abs{z_j}^2 = 1 \Bigr\} .
\end{equation}
Hyperquadrics are the model hypersurfaces for Levi-nondegenerate
surfaces.
The hyperquadrics are also the flat models in CR geometry from the point of view of Riemannian geometry.

The CR geometry problem of maps into hyperquadrics has a long history beginning with
Webster who in 1978~\cite{Webster:1978} proved that every algebraic
hypersurface can be embedded into a hyperquadric for large enough $a$ and
$b$.  On the other hand 
Forstneri\v{c}~\cite{Forstneric86} proved that in general a CR submanifold will not map
to a finite dimensional hyperquadric.
Thus not every CR submanifold can be realized as a submanifold
of the flat model (i.e. the hyperquadric).

The mapping problem has been studied extensively for proper maps between
balls, where the
related CR question is to classify the CR maps between spheres.
That is, suppose $f \colon \bB^n \to \bB^N$ is a proper holomorphic map.
Two such maps $f$ and $g$ are \emph{spherically equivalent} if
there exist automorphisms $\tau \in \operatorname{Aut}(\bB^n)$
and $\chi \in \operatorname{Aut}(\bB^N)$ such that $f \circ \tau = \chi
\circ g$.
There has been considerable progress in the classification of such maps,
especially for small codimensional
cases; see for example
\cites{Alexander77,Faran82,Forstneric89,DAngelo88,DAngelo13,HuangJiYin14}
and the many references within.

In particular, Forstneri\v{c}~\cite{Forstneric89} proved that given
sufficient boundary regularity, the map is rational and the degree
bounded in terms of the dimensions only.  The degree is an invariant
under spherical equivalence and the sharp bounds for the degree were
conjectured by D'Angelo~\cite{DKR}.

In this article we introduce a new biholomorphic invariant that is far more general than 
degree.  When the map is rational and the domain and target are balls
(possibly of nonzero signature), we obtain invariants via generic initial
ideals from commutative algebra.  For arbitrary domains and maps, we
generalize the smallest degree part of the ideal.

The techniques center on the idea of generic initial monomial ideals; see
e.g.\ Green~\cite{Green:gin}.  Given a homogoeneous ideal $\sI$, 
the monomial ideal $\operatorname{in}(\sI)$ is generated by the initial monomials of elements of $\sI$ (see section 3 for precise definitions).  
The \emph{generic initial ideal}, denoted $\operatorname{gin}(\sI)$, is an initial monomial ideal after
precomposing $\sI$ with a generic invertible linear map.
Grauert was the first to introduce generic initial ideals to several complex variables
(see \cite{Grauert}), and they have been used extensively
for understanding singularities of varieties.

For rational maps of balls of the form $\frac{f}{g}$, we homogenize $f$ and $g$, and we then look at
the ideal generated by the components.  We prove that the generic initial
ideal generated by the homogenizations of $f$ and $g$ is invariant under spherical equivalence. More precisely, the first main result is the following.

\begin{thm}
Suppose $f_1 \colon \bB^n \to \bB^N$ and
$f_2 \colon \bB^n \to \bB^N$ are rational proper maps that are spherically
equivalent.  Let $F_1$ and $F_2$ be the respective homogenizations.  Then
\begin{equation}
\operatorname{gin}\bigl(\sI(F_1)\bigr) = 
\operatorname{gin}\bigl(\sI(F_2)\bigr) .
\end{equation}
\end{thm}

In section \ref{section:ratmaps}, we prove this result, and compute this new
invariant for a number of well-known examples. In particular we show how this invariant distinguishes many of these maps. In section \ref{section:quotientinvration}, we obtain a further
invariant by considering the holomorphic decomposition of the quotient
\begin{equation} \label{eq:quotient}
\frac{\snorm{f(z)}^2-\sabs{g(z)}^2}{\snorm{z}^2-1} .
\end{equation}

To generalize these results to arbitrary holomorphic maps we improve
a result proved in Grundmeier-Lebl-Vivas~\cite{GLV} that
itself is a version of Galligo's theorem for vector subspaces of $\sO(U)$.
Given a subspace $X \subset \sO(U)$, we precompose with a generic affine
map $\tau$, and
define the generic initial monomial subspace as the
space spanned by the initial monomials of elements of $X \circ \tau$, denoted $\operatorname{gin}(X)$.
Let $X$ be the affine span of the components of a holomorphic
map $F \colon U \to \C^N$, denoted by $\afspan{F}$. See section \ref{section:affinespans} for precise definitions.
We obtain an invariant under biholomorphic transformations of $U$
and invertible linear fractional transformations of $\C^N$.
By using affine rather than just linear maps we obtain invariants of the map
without having a distinguished point. The second main result says this is,
indeed, an invariant.

\begin{thm}
Let $M, M' \subset \C^n$ be connected real-analytic CR submanifolds.
Suppose $F \colon M \to Q(a,b)$ and
$G \colon M' \to Q(a,b)$ are real-analytic CR maps
equivalent in the sense that
there exists a real-analytic CR isomorphism $\tau \colon M' \to M$
and a linear fractional automorphism $\chi$ of $Q(a,b)$ such that
\begin{equation}
F \circ \tau = \chi \circ G .
\end{equation}
Then
\begin{equation}
\operatorname{gin}(\afspan{F}) =
\operatorname{gin}(\afspan{G}) .
\end{equation}
\end{thm}

In sections \ref{section:affinespans} and \ref{section:gins}, we give
precise definitions and lemmas. In particular we prove the extended version of Galligo's theorem for our setting. In section \ref{section:ginsasinvar}, we prove the second major theorem of this paper, and we give examples where we compute this invariant. Finally, in section \ref{section:quotient}
we define the generic initial
subspace for the analogue of the quotient \eqref{eq:quotient} 
for CR maps between spheres and hyperquadrics.

The authors would like to acknowledge John D'Angelo for many conversations
on the subject.

%%%%%%%%%%%%%%%%%%%%%%%%%%%%%%%%%%%%%%%%%%%%%%%%%%%%%%%%%%%%%%%%%%%%%%%%%%%%

\section{The projective setting}
\label{section:projset}

Before we work with arbitrary maps, we consider the special case
of rational maps between hyperquadrics and spheres.  In this case we 
directly apply the standard theory of generic initial ideals to obtain invariants.

For $z,w \in \C^{a+b}$ define
\begin{equation}
\linnprod{z}{w}_b \overset{\text{def}}{=}
\linnprod{I_b z}{w} =
-
\sum_{j=1}^{b} z_j \bar{w}_j 
+
\sum_{j=b+1}^{a+b} z_j \bar{w}_j 
\qquad \text{and} \qquad
\norm{z}_b^2 \overset{\text{def}}{=} 
\linnprod{z}{z}_b .
\end{equation}
By $I_b$ we mean the $(a+b)\times(a+b)$ diagonal matrix with $b$
$(-1)$'s and $a$ $1$'s on the diagonal.  We use the same definition for
the homogeneous case when the index on the variables starts with a zero.
In this case the subscript still refers to the number of negatives.  That is,
for $Z = (Z_0,\ldots,Z_{a+b}) \in \C^{a+b+1}$ we write
\begin{equation}
\linnprod{Z}{W}_{b+1} \overset{\text{def}}{=}
\linnprod{I_{b+1} Z}{W} =
-
\sum_{j=0}^{b} Z_j \bar{W}_j 
+
\sum_{j=b+1}^{a+b} Z_j \bar{W}_j 
\qquad \text{and} \qquad
\norm{Z}_{b+1}^2 \overset{\text{def}}{=} 
\linnprod{Z}{Z}_{b+1} .
\end{equation}

Let us homogenize $Q(a,b)$.  We add a variable $Z_0$ and work with
the homogeneous coordinates $[Z_0,Z_1,\ldots,Z_{n}]$ in $\bP^n$.
Homogenize the equation above to obtain
\begin{equation} \label{eq:defhq}
HQ(a,b+1) \overset{\text{def}}{=} \left\{ Z \in \bP^{a+b} : \norm{Z}_{b+1}^2 = 0 \right\} .
\end{equation}
If we think of $\C^{a+b} \subset \bP^{a+b}$, then $Q(a,b)$ is a subset of
$\bP^{a+b}$ and
$HQ(a,b+1)$ is the closure of $Q(a,b)$ in $\bP^{a+b}$.

Automorphisms of $\bP^{a+b}$ can be represented as invertible
linear maps on $\C^{a+b+1}$.
The automorphisms of $HQ(a,b+1)$ are those linear maps $T$
that preserve the form in \eqref{eq:defhq} up to a real scalar $\lambda \not= 0$, that is
\begin{equation}
\norm{T Z}_{b+1}^2 = \lambda \norm{Z}_{b+1}^2 .
\end{equation}
If we represent $T$ as a matrix, then the condition above is $T^* I_{b+1}
T = \mu I_b$, where $\mu$ is $\pm\sqrt{\abs{\lambda}}$.
If $a \not= b+1$, then $\lambda > 0$.  If $\lambda < 0$, then the
automorphism swaps the sides of $HQ(a,b+1)$.
Write the set
of the corresponding automorphisms of $\bP^{a+b}$ as $\operatorname{Aut}\bigl(HQ(a,b+1)\bigr)$.
As usual the point $Z \in \bP^{a+b}$ is an equivalence class of points of $\C^{a+b+1}$ up to
complex multiple.  Thus (as long as $a \not= b+1$)
the group $\operatorname{Aut}\bigl(HQ(a,b+1)\bigr)$ is the group $SU(a,b+1)
/ K$ where $K$ is the subgroup of matrices $\zeta I$ where $\zeta^{a+b+1} =
1$.  If $a = b+1$, we need to include the matrix that switches sides.
The important fact for us is that automorphisms are
represented by matrices.

Fix $(a,b)$ and $(A,B)$.
A rational map $F \colon \bP^{a+b} \dashrightarrow \bP^{A+B}$ is
represented by a homogeneous polynomial map of $\C^{a+b+1}$ to $\C^{A+B+1}$.
The equivalence we wish to consider is the following.
Two
rational maps $F \colon \bP^{a+b} \dashrightarrow \bP^{A+B}$ and
$G \colon \bP^{a+b} \dashrightarrow \bP^{A+B}$ are equivalent if there
exists $\tau \in \operatorname{Aut}\bigl(\bP^{a+b+1}\bigr)$ and
$\chi \in \operatorname{Aut}\bigl(\bP^{A+B+1}\bigr)$ such that
\begin{equation}
F \circ \tau = \chi \circ G .
\end{equation}
In the applications we will have
$\tau \in \operatorname{Aut}\bigl(HQ(a,b+1)\bigr)$ and
$\chi \in \operatorname{Aut}\bigl(HQ(A,B+1)\bigr)$.

If $f \colon \bB^n \to \bB^N$ is a rational proper map of balls,
its homogenization $F \colon \bP^{n} \dashrightarrow \bP^{N}$
takes $HQ(n,1)$ to $HQ(N,1)$.
The equivalence on such maps $F$
using the groups
$\operatorname{Aut}\bigl(HQ(n,1)\bigr)$ and
$\operatorname{Aut}\bigl(HQ(N,1)\bigr)$
is precisely the standard spherical equivalence on $f$; that is,
$f \colon \bB^n \to \bB^N$ and
$g \colon \bB^n \to \bB^N$ are spherically equivalent
if there exist automorphisms 
$\tau$ and $\chi$ of $\bB^n$ and $\bB^N$ respectively such that
$f \circ \tau = \chi \circ g$.

%%%%%%%%%%%%%%%%%%%%%%%%%%%%%%%%%%%%%%%%%%%%%%%%%%%%%%%%%%%%%%%%%%%%%%%%%%%%

\section{Generic initial ideals and rational maps}
\label{section:ratmaps}

In this section we briefly introduce the relevant definitions and results from commutative algebra.
We use the setup from
Green \cite{Green:gin}, and in a later section we use the techniques developed by the authors in
\cite{GLV}.  To be consistent with these two papers, we use the
slightly unusual monomial ordering as used by Green.  

Let $Z_0,Z_1,\ldots,Z_n$ denote our variables.  Given a multi-index 
$\alpha \in \N_0^{n+1}$ we write
$Z^\alpha$ to mean
$Z_0^{\alpha_0}Z_1^{\alpha_1}\cdots
Z_n^{\alpha_n}$ as usual, and let $\abs{\alpha} = \alpha_0 + \cdots + \alpha_n$
denote the total degree.  A
\emph{multiplicative monomial order}
is a total ordering on all monomials $Z^\alpha$, such that
\begin{enumerate}[(i)]
\item
$Z_0 > Z_1 > \cdots > Z_n$,
\item $Z^\alpha > Z^\beta ~ \Rightarrow ~ Z^\gamma Z^\alpha >
Z^\gamma Z^\beta$,
\item
$\abs{\alpha} < \abs{\beta} ~\Rightarrow~ Z^\alpha  > Z^\beta$.
\end{enumerate}
Such orderings are not unique.  Using the so-called graded reverse lexicographic
ordering when $n=2$, we obtain
\begin{equation}
1 > Z_0 > Z_1 > Z_2 > Z_0^2 > Z_0Z_1 > Z_1^2 > Z_0Z_2 > Z_1 Z_2 > Z_2^2 >
\cdots .
\end{equation}
%An ordering with $Z_0Z_2 > Z_1^2$ would also be a multiplicative monomial order.

Fix a certain monomial order.
Given a set $S$ of monomials, then the \emph{initial monomial}
is the maximal monomial in $S$ according to the ordering.  
For a homogeneous polynomial $P$, write
\begin{equation}
\operatorname{in}(P) \overset{\text{def}}{=} \text{initial monomial of $P$.}
\end{equation}
Let $\sI$ be a homogeneous ideal.  Define
the \emph{initial monomial ideal} $\operatorname{in}(\sI)$
as the smallest ideal such that $z^\alpha \in \operatorname{in}(\sI)$
whenever $z^\alpha = \operatorname{in}(P)$ for some $P \in \sI$.

Galligo's theorem implies that for an open dense set of linear maps $T$,
$\operatorname{in}(\sI) = \operatorname{in}(\sI \circ T)$.
Let $T$ be such a generic linear map and define
\begin{equation}
\operatorname{gin}(\sI) \overset{\text{def}}{=} \operatorname{in}(\sI \circ
T) .
\end{equation}

Finally, for a rational map $F \colon \bP^{n} \dashrightarrow \bP^{N}$ define
\begin{equation}
\sI(F) = \text{ideal generated by components of $F$ in homogeneous
coordinates.}
\end{equation}

\begin{prop}
Suppose $F \colon \bP^{n} \dashrightarrow \bP^{N}$ and
$G \colon \bP^{n} \dashrightarrow \bP^{N}$ are equivalent in the sense
that there exist
$\tau \in \operatorname{Aut}\bigl(\bP^n\bigr)$ and
$\chi \in \operatorname{Aut}\bigl(\bP^N\bigr)$ such that
\begin{equation}
F \circ \tau = \chi \circ G .
\end{equation}
Then
\begin{equation}
\operatorname{gin}\bigl(\sI(F)\bigr) = 
\operatorname{gin}\bigl(\sI(G)\bigr) .
\end{equation}
\end{prop}

\begin{proof}
%The proof is obvious
By definition of the generic initial ideal,
$\operatorname{gin}\bigl(\sI(F \circ \tau)\bigr) = 
\operatorname{gin}\bigl(\sI(F)\bigr)$,
as $\tau$ is an invertible linear map.
The ideal generated by
linear combinations of $G$ is the same as the one generated by $G$,
and $\chi$ is also an invertible linear map.  So,
$\operatorname{gin}\bigl(\sI(\chi \circ G)\bigr) = 
\operatorname{gin}\bigl(\sI(G)\bigr)$.
\end{proof}

In the important special case of rational proper maps of balls we get the first main theorem from the introduction.

\begin{cor}
Suppose $f \colon \bB^n \to \bB^N$ and
$g \colon \bB^n \to \bB^N$ are rational proper maps that are spherically
equivalent.  Let $F$ and $G$ be the respective homogenizations.  Then
\begin{equation}
\operatorname{gin}\bigl(\sI(F)\bigr) = 
\operatorname{gin}\bigl(\sI(G)\bigr) .
\end{equation}
\end{cor}

We close this section with several examples.  The following computations are done with 
\texttt{Macaulay2}~\cite{M2} with the \texttt{GenericInitialIdeal}
package using the standard graded reverse lex ordering unless stated otherwise.
One advantage of using gins as invariants
is that they are very simple to compute using
computer algebra systems.

\begin{example}
Generic initial
ideals distinguish all the maps from $\bB^2$ to $\bB^3$.
Faran~\cite{Faran82}
proved that all maps sufficiently smooth up to the boundary are
spherically equivalent to one of the following 4 maps:
\begin{enumerate}[(i)]
\item $(z_1,z_2) \mapsto (z_1,z_2,0)$,
\item $(z_1,z_2) \mapsto (z_1,z_1z_2,z_2^2)$,
\item $(z_1,z_2) \mapsto (z_1^2,\sqrt{2}z_1z_2,z_2^2)$,
\item $(z_1,z_2) \mapsto (z_1^3,\sqrt{3}z_1z_2,z_2^3)$.
\end{enumerate}

Let us list the results in a table.  The first column gives the
map itself.  The second column gives the map homogenized by
adding the $Z_0$ variable.  The last column gives the
generators of the generic initial ideal.
\begin{center}
\begin{tabular}{l|l|l}
Map & Homogenized & gin
\\
\hline
\hline
$(z_1,z_2,0)$ &
$(Z_0,Z_1,Z_2,0)$ &
$(Z_0,Z_1,Z_2)$
\\
\hline
$(z_1,z_1z_2,z_2^2)$ &
$(Z_0^2,Z_1Z_0,Z_1Z_2,Z_2^2)$ &
$(Z_0^2,Z_0Z_1,Z_1^2,Z_0Z_2,Z_1Z_2^2)$
\\
\hline
$(z_1^2,\sqrt{2}z_1z_2,z_2^2)$ &
$(Z_0^2,Z_1^2,\sqrt{2}Z_1Z_2,Z_2^2)$ &
$(Z_0^2,Z_0Z_1,Z_1^2,Z_0Z_2,Z_2^3)$
\\
\hline
$(z_1^3,\sqrt{3}z_1z_2,z_2^3)$ &
$(Z_0^3,Z_1^3,\sqrt{3}Z_1Z_2Z_0,Z_2^3)$ &
$(Z_0^3, Z_0^2Z_1, Z_0Z_1^2, Z_0^2Z_2,$\\ 
& & $\phantom{(}Z_1^4, Z_1^3Z_2, Z_0Z_1Z_2^2, Z_1^2Z_2^2,$\\
& & $\phantom{(}Z_0Z_2^4, Z_1Z_2^4, Z_2^5)$
\end{tabular}
\end{center}
Notice that the generic initial ideals are all different.  In particular,
we distinguish the two very similar second degree maps
using the third degree part of the ideal.
\end{example}

\begin{example}
In \cite{FHJZ} it was proved that the proper map of $\bB_2$ to $\bB_4$ given by
\begin{equation}
(z_1,z_2) \mapsto
\left(
z_1^2,\sqrt{2}z_1z_2,\frac{z_2^2(z_1-a)}{1-\bar{a}z_1},\frac{z_2^3\sqrt{1-\abs{a}^2}}{1-\bar{a}z_1}
\right)
\end{equation}
is spherically equivalent to a polynomial map if and only if $a=0$.  This
map is the homogeneous second degree map from Faran's theorem above with the
last component tensored by an automorphism of the ball $\bB_2$ taking $a$ to
0.  If $a=0$, then the generic
initial ideal is
\begin{equation}
(Z_0^3,Z_0^2Z_1,Z_0Z_1^2,Z_1^3,Z_0^2Z_2,Z_0Z_1Z_2^2,Z_1^2Z_2^2,Z_0Z_2^4) .
\end{equation}
If $a=\frac{1}{2}$, we obtain a different generic initial ideal
\begin{equation}
(Z_0^3,Z_0^2Z_1,Z_0Z_1^2,Z_1^3,Z_0^2Z_2,Z_0Z_1Z_2^2,Z_1^2Z_2^2,Z_0Z_2^3) .
\end{equation}
\end{example}

\begin{example} In \cite{Lebl:hq23}
the second author
classified 
all the maps from the two sphere, $Q(2,0)$, to the hyperquadric $Q(2,1)$
(see also Reiter~\cite{Reiter:class} for a different approach to the classification).  There are
7 equivalence classes of maps:
\begin{enumerate}[(i)]
\item \label{hqclass1}
$(z_1,z_2) \mapsto (0,z_1,z_2)$,
\item \label{hqclass2}
$(z_1,z_2) \mapsto (z_2^2,z_1^2,\sqrt{2}\,z_2)$,
\item \label{hqclass3}
$(z_1,z_2) \mapsto \left( \frac{z_2}{z_1^2}, \frac{1}{z_1}, \frac{z_2^2}{z_1^2}  \right)$,
\item \label{hqclass4}
$(z_1,z_2) \mapsto \left(
\frac{z_1^2-\sqrt{3}\,z_1z_2+z_2^2-z_1}{z_2^2+z_1+\sqrt{3}\,z_2 - 1} ,
\frac{z_1^2+\sqrt{3}\,z_1z_2+z_2^2-z_1}{z_2^2+z_1+\sqrt{3}\,z_2 - 1} ,
\frac{z_2^2+z_1-\sqrt{3}\,z_2 - 1}{z_2^2+z_1+\sqrt{3}\,z_2 - 1}
\right)$,
\item \label{hqclass5}
$(z_1,z_2) \mapsto \left(
\frac{\sqrt[4]{2}(z_1z_2+iz_1)}{z_2^2+\sqrt{2}\,iz_2+1} ,
\frac{\sqrt[4]{2}(z_1z_2-iz_1)}{z_2^2+\sqrt{2}\,iz_2+1} ,
\frac{z_2^2-\sqrt{2}\,iz_2+1}{z_2^2+\sqrt{2}\,iz_2+1} 
\right)$,
\item \label{hqclass6}
$(z_1,z_2) \mapsto
\left( 
\sqrt{3}\frac{z_2z_1^2-z_2}{3z_1^2+1} ,
\frac{2z_2^3}{3z_1^2+1} ,
\frac{z_1^3+3z_1}{3z_1^2+1}
\right)$,
\item \label{hqclass7}
$(z_1,z_2) \mapsto \bigl(g(z_1,z_2),g(z_1,z_2),1\bigr)$ for an arbitrary holomorphic function $g$.
\end{enumerate}
Note that the maps are written slightly differently from \cite{Lebl:hq23} as we are ordering the negative components first.

We list the generic initial ideal for each of the first 6 classes of maps.  The final 7th class is already distinguished as it is not transversal; that is, it maps an open neighborhood into $Q(2,1)$.
\begin{center}
\begin{tabular}{l|l|l}
Map & Homogenized & gin
\\
\hline
\hline
\eqref{hqclass1} &
$(Z_0,0,Z_1,Z_2)$ &
$(Z_0,Z_1,Z_2)$
\\
\hline
\eqref{hqclass2} &
$(Z_0^2,Z_2^2,Z_1^2,\sqrt{2}Z_2Z_0)$ &
$(Z_0^2, Z_0Z_1, Z_1^2, Z_0Z_2, Z_1Z_2^2, Z_2^3)$
\\
\hline
\eqref{hqclass3} &
$(Z_1^2,Z_2Z_0,Z_1Z_0,Z_2^2)$ &
$(Z_0^2, Z_0Z_1, Z_1^2, Z_0Z_2, Z_1Z_2^2)$
\\
\hline
\eqref{hqclass4} &
$(Z_2^2+Z_1Z_0+\sqrt{3}Z_2Z_0-Z_0^2,$ &
$(Z_0^2, Z_0Z_1, Z_1^2, Z_0Z_2, Z_1Z_2^2)$

\\
& $\phantom{(} Z_1^2-\sqrt{3}Z_1Z_2+Z_2^2-Z_1Z_0,$ & \\
& $\phantom{(} Z_1^2+\sqrt{3}Z_1Z_2+Z_2^2-Z_1Z_0,$ & \\
& $\phantom{(} Z_2^2+Z_1Z_0-\sqrt{3}Z_2Z_0-Z_0^2)$ &
\\
\hline
\eqref{hqclass5} &
$(Z_2^2+\sqrt{2}iZ_2Z_0+Z_0^2,$ &
$(Z_0^2, Z_0Z_1, Z_1^2, Z_0Z_2, Z_1Z_2^2)$
\\
& $\phantom{(} \sqrt[4]{2}(Z_1Z_2+iZ_1Z_0),$ & \\
& $\phantom{(} \sqrt[4]{2}(Z_1Z_2-iZ_1Z_0),$ & \\
& $\phantom{(} Z_2^2-\sqrt{2}iZ_2Z_0+Z_0^2)$ &
\\
\hline
\eqref{hqclass6} &
$(3Z_1^2Z_0+Z_0^3, \sqrt{3}(Z_2Z_1^2-Z_2Z_0^2),$ &
$(Z_0^3,Z_0^2Z_1,Z_0Z_1^2,Z_0^2Z_2,$
\\
& $\phantom{(} 2Z_2^3, Z_1^3+3Z_1Z_0^2)$ &
$\phantom{(}Z_1^4,Z_0Z_1Z_2^2,Z_1^2Z_2^2,Z_1^3Z_2,$ \\
& & $\phantom{(}Z_0Z_2^4,Z_1Z_2^4,Z_2^5)$ 
\end{tabular}
\end{center}
We distinguish some maps, but not all.
It is to be expected that not all maps can be distinguished.
After all, computing
a generic initial ideal throws away much information about the map.
It should be noted that we only looked at small degree examples,
where the number of different possible gins is limited.
\end{example}

\begin{example}
The particular monomial ordering may make a difference, even in simple
situations.  It might be possible to tell some ideals (and hence maps) apart
using one ordering, but not using another.  Let us give a very simple
example.  In $\bP^2$, the
ideal generated by $(Z_0^2,Z_1Z_2)$ has the gin
\begin{equation}
\begin{aligned}
& (Z_0^2,Z_0Z_1,Z_1^3) & & \quad \text{in graded reverse lex ordering,}\\
& (Z_0^2,Z_0Z_1,Z_0Z_2^2,Z_1^4) & & \quad \text{in graded lex ordering.}
\end{aligned}
\end{equation}
However, the ideal generated by $(Z_0^2,Z_1^2)$ has the same gin in both
orderings:
\begin{equation}
(Z_0^2,Z_0Z_1,Z_1^3) .
\end{equation}
\end{example}

%%%%%%%%%%%%%%%%%%%%%%%%%%%%%%%%%%%%%%%%%%%%%%%%%%%%%%%%%%%%%%%%%%%%%%%%%%%%

\section{Invariants of the quotient for rational maps}
\label{section:quotientinvration}

Let
$F \colon \bP^{a+b} \dashrightarrow \bP^{A+B}$ be a rational map
that takes $HQ(a,b+1)$ to
$HQ(A,B+1)$, where defined.  Identify $F$ with the
homogeneous polynomial map taking $\C^{a+b+1}$ to $\C^{A+B+1}$.
Write
\begin{equation}
\norm{F(Z)}_{B+1}^2 = \norm{Z}_{b+1}^2 q(Z,\bar{Z}) ,
\end{equation}
where the quotient $q$ is a bihomogeneous polynomial in $Z$.
Find the holomorphic decomposition (see p. 101 of \cite{DAngelo:book}) of $q$ as
\begin{equation}
q(Z,\bar{Z}) = \norm{h_+(Z)}^2-\norm{h_-(Z)}^2 ,
\end{equation}
where $h_+$ and $h_-$ are homogeneous holomorphic polynomial maps.
Write
\begin{equation}
H(q) = \{ h_+,h_- \}
\end{equation}
for the set of functions in the holomorphic decomposition of $q$.
The holomorphic decomposition is not unique.  However, the linear span
of $H(q)$ is unique, and therefore the ideal generated by $H(q)$ is unique.
We thus study the ideal $\sI\bigl(H(q)\bigr)$ and
furthermore the generic initial ideal
$\operatorname{gin}\bigl(\sI\bigl(H(q)\bigr)\bigr)$.

The reason for looking at $q$ is that it may reveal further information about
the map that is not immediately visible from $F$.  For example the quotient $q$
was critically used in \cite{DKR} for the degree estimates problem.
The number of functions
in the decomposition of $q$ is often larger than the number of
components of $F$ as many of these components may cancel out
once multiplied with $\norm{Z}_{b+1}^2$. 

Suppose $F \colon \bP^{a+b} \dashrightarrow \bP^{A+B}$ and
$G \colon \bP^{a+b} \dashrightarrow \bP^{A+B}$ are equivalent as before; i.e.
$F \circ \tau = \chi \circ G$, where $\tau$ and $\chi$ are automorphisms
preserving $HQ(a,b+1)$ and $HQ(A,B+1)$ respectively.
Let $q_F$ and $q_G$ be the corresponding
quotients of $F$ and $G$ respectively.  We regard $\tau$ and $\chi$ as linear
maps, which we also rescale, such that $\norm{\tau Z}_{b+1}^2 = \pm \norm{Z}_{b+1}^2$
for $Z \in \C^{a+b+1}$ and
$\norm{\chi W}_{B+1}^2 = \pm \norm{W}_{B+1}^2$ for $W \in \C^{A+B+1}$.  The
$\pm$ is there in case $a=b+1$ or $A=B+1$ and the sides are swapped,
otherwise it would be a $+$.

As
$F$ and $G$ are equivalent, $\norm{F(\tau Z)}_{B+1}^2 =
\norm{\chi \circ G(Z)}_{B+1}^2$.
Then
\begin{multline}
\norm{Z}_{b+1}^2 q_F(\tau Z,\overline{\tau Z})
=
\pm \norm{\tau Z}_{b+1}^2 q_F(\tau Z,\overline{\tau Z})
=
\pm \norm{F(\tau Z)}_{B+1}^2
=
\\
=
\pm \norm{\chi \circ G}_{B+1}^2
=
\pm \norm{G(Z)}_{B+1}^2
=
\pm \norm{Z}_{b+1}^2 q_G(Z,\bar{Z}) .
\end{multline}
In other words
$q_F(\tau Z,\overline{\tau Z}) = \pm q_G(Z,\bar{Z})$.

\begin{prop}
Suppose $F \colon \bP^{a+b} \dashrightarrow \bP^{A+B}$ and
$G \colon \bP^{a+b} \dashrightarrow \bP^{A+B}$ are equivalent in the sense
that there exist
$\tau \in \operatorname{Aut}\bigl(HQ(a,b+1)\bigr)$ and
$\chi \in \operatorname{Aut}\bigl(HQ(A,B+1)\bigr)$ such that
\begin{equation}
F \circ \tau = \chi \circ G .
\end{equation}
Let $q_F$ and $q_G$ be the corresponding quotients.
Then
\begin{equation}
\operatorname{gin}\Bigl(\sI\bigl(H(q_F)\bigr)\Bigr) = 
\operatorname{gin}\Bigl(\sI\bigl(H(q_G)\bigr)\Bigr) .
\end{equation}
\end{prop}

\begin{proof}
Above we proved
$q_F(\tau Z,\overline{\tau Z}) = \pm q_G(Z,\bar{Z})$.  
As we are talking about gins, the $\tau$ is not relevant.  The $\pm$
does not change the components of the holomorphic decomposition.
\end{proof}

Again, in the important special case of rational proper maps of balls we get:

\begin{cor}
Suppose $f \colon \bB^n \to \bB^N$ and
$g \colon \bB^n \to \bB^N$ are rational proper maps that are spherically
equivalent.  Let $F$ and $G$ be the respective homogenizations and
let $q_F$ and $q_G$ be the corresponding quotients.
Then
\begin{equation}
\operatorname{gin}\Bigl(\sI\bigl(H(q_F)\bigr)\Bigr) = 
\operatorname{gin}\Bigl(\sI\bigl(H(q_G)\bigr)\Bigr) .
\end{equation}
\end{cor}

\begin{example}
Let us consider the Faran map $(z_1,z_2) \mapsto
(z_1^3,\sqrt{3}z_1z_2,z_2^3)$ that takes $\bB^2$ to $\bB^3$.  We compute
the quotient
\begin{equation}
\frac{\abs{z_1^3}^2+\abs{\sqrt{3}z_1z_2}^2+\abs{z_2^3}^2-1}{\abs{z_1}^2+\abs{z_2}^2-1}
=
\abs{z_1^2}^2+\abs{z_2^2}^2-\abs{z_1z_2}^2+\abs{z_1}^2+\abs{z_2}^2+1 .
\end{equation}
Bihomogenizing the quotient, we obtain all the degree-two monomials in
$H(q)$.  Therefore the gin is generated by all the degree-two monomials:
\begin{equation}
\operatorname{gin}\Bigl(\sI\bigl(H(q)\bigr)\Bigr) =
(Z_0^2, Z_0Z_1, Z_1^2, Z_0Z_2, Z_1 Z_2, Z_2^2).
\end{equation}
\end{example}

%%%%%%%%%%%%%%%%%%%%%%%%%%%%%%%%%%%%%%%%%%%%%%%%%%%%%%%%%%%%%%%%%%%%%%%%%%%%

\section{Affine spans and automorphisms of the target}
\label{section:affinespans}

To deal with nonrational maps we 
switch to the affine setting.  We work in $\C^{N}$
and treat it as a subset of $\bP^{N}$.  From now on
$z = (z_1,\ldots,z_{N})$ will be the inhomogeneous coordinates on $\bP^{N}$,
that is, coordinates on $\C^{N}$.  In other words, we fix a specific
embedding
$\iota \colon \C^N \to \bP^N$.

For equivalence of maps $F \colon U \to \C^N$ we consider the target
automorphisms to be the linear fractional automorphisms of $\operatorname{Aut}(\bP^N)$
and consider $\C^N \subset \bP^N$ as above.  Take
$\chi \in \operatorname{Aut}(\bP^N)$, 
$F \colon U \to \C^N$, and $G \colon U \to \C^N$.  When we write
an equation such as $\chi \circ G = F$, we mean
$\chi \circ \iota \circ G = \iota \circ F$, where $\iota$ is the embedding above.
That is, we consider $F$ to be valued in $\bP^N$
using the embedding $\iota \colon \C^N \to \bP^N$.
Note that after composing with automorphisms of $\bP^N$ given as linear
fractional maps,
it may happen that the new map has poles in $U$.
In fact, when looking at linear fractional automorphisms of
$Q(a,b)$, there will in general be poles that intersect $Q(a,b)$.

The reason for not simply working in projective space is that the components
in nonrational holomorphic maps $F \colon U \to \C^N$
cannot be homogenized; the degree is unbounded.
For this same reason, we have to work with
affine span instead of just linear span.

\begin{defn}
Given a collection of holomorphic functions $\sF$, the \emph{affine span}
of $\sF$ is
\begin{equation}
\afspan{\sF} = \operatorname{span} \bigl( \sF \cup \{ 1 \} \bigr) .
\end{equation}
\end{defn}
If $F$ is a map, then
$\afspan{F}$ is the affine span of the components of $F$.
Let $X \subset \sO(U)$ be a vector subspace and $\varphi \in \sO(U)$.
Denote by $\varphi X$ the vector subspace obtained by multiplying every element in $X$
by $\varphi$.

\begin{lemma} \label{lemma:targetspan}
Suppose $F \colon U \subset \C^n \to \C^{N}$ and
$G \colon U \subset \C^n \to \C^{N}$ are such that
there exists %linear fractional automorphism
$\chi \in
\operatorname{Aut}(\bP^N)$
such that $F = \chi \circ G$.
Then there exists $\varphi \in \afspan{G}$ such that
\begin{equation}
\afspan{G} = \varphi \afspan{F} .
\end{equation}
\end{lemma}

\begin{proof}
Let $\chi(w) = \frac{P(w)}{Q(w)}$ be the linear fractional automorphism
such that $F(z) = \frac{P \circ G(z)}{Q \circ G(z)}$ for $z \in U$.  Here $P
\colon \C^{N} \to
\C^{N}$ and $Q \colon \C^{N} \to \C$ are affine maps.
Therefore
$( Q \circ G ) F = P \circ G$.  Therefore the components of $( Q \circ G ) F$ are in the affine
span of $G$.  Also $Q \circ G$ is in the affine span of $G$.
The span of the components of $( Q \circ G ) F$ and the function $(Q \circ G)$
is exactly $(Q \circ G) \afspan{F}$, and therefore
$(Q \circ G) \afspan{F} \subset \afspan{G}$.

Clearly $\dim \afspan{F} \leq \dim \afspan{G}$.  By inverting $\chi$ and
applying the argument in reverse we get $\dim \afspan{G} \leq \dim
\afspan{F}$.  Therefore,
$(Q \circ G) \afspan{F} = \afspan{G}$.
\end{proof}

The point of the lemma is to eliminate the target
automorphism by looking at the
vector space generated by $F_1,\ldots,F_N$ and 1, up to multiplication by a
scalar-valued function.  To find invariants of the map $F$ we need to find
invariants of this new object under biholomorphisms of the source.

%%%%%%%%%%%%%%%%%%%%%%%%%%%%%%%%%%%%%%%%%%%%%%%%%%%%%%%%%%%%%%%%%%%%%%%%%%%%

\section{Generic initial vector space} \label{section:gins}

In this section we define the \emph{gin} of
a vector subspace of holomorphic maps.
The idea is to generalize the generic initial ideals 
to the setting of vector spaces of holomorphic functions using affine maps
instead of linear maps.
We generalize the setup from 
Green \cite{Green:gin}, using the techniques developed by the authors in
\cite{GLV}.

Let monomial order be defined on $z_1,\ldots,z_n$ in the same way as above.
That is, given a multi-index 
$\alpha \in \N_0^n$ we write
$z^\alpha$ to mean
$z_1^{\alpha_1}z_2^{\alpha_2}\cdots
z_n^{\alpha_n}$, and $\abs{\alpha} = \alpha_1 + \cdots + \alpha_n$ is the
total degree.  We require the order to be multiplicative:
\begin{enumerate}[(i)]
\item
$z_1 > z_2 > \cdots > z_n$,
\item $z^\alpha > z^\beta ~ \Rightarrow ~ z^\gamma z^\alpha >
z^\gamma z^\beta$,
\item
$\abs{\alpha} < \abs{\beta} ~\Rightarrow~ z^\alpha  > z^\beta$.
\end{enumerate}
In the following definitions we fix a point.  We assume 
this point is $0 \in \C^n$, although any point can be used after translation.

\begin{defn}
Fix a monomial order.  By an initial monomial from a collection we mean
the
largest monomial in the order.  Given a Taylor series $T$, the initial
monomial of $T$ is the largest monomial in $T$ with a nonzero coefficient.
Suppose $U \subset \C^n$ and $0 \in U$.
For $f \in \sO(U)$, let $T_0f$ be the Taylor series for $f$
at $0$ and
\begin{equation}
\operatorname{in}(f) \overset{\text{def}}{=} \text{initial monomial of $T_0f$.}
\end{equation}
Given $X \subset \sO(U)$ a vector subspace, define
the \emph{initial monomial subspace}
\begin{equation}
\operatorname{in}(X) \overset{\text{def}}{=}
\operatorname{span} \{ z^\alpha : z^\alpha = \operatorname{in}(f) \text{ for some } f \in
X \} .
\end{equation}
A subspace $X \subset \sO(U)$ is a \emph{monomial subspace} if
$X$ admits an algebraic basis consisting of monomials.
\end{defn}

If $X$ is a monomial subspace then the basis of monomials
defining it is unique.  Therefore, there is a one-to-one equivalence
between monomial subspaces and subsets of the set of all monomials.

A set $S$ of monomials is \emph{affine-Borel-fixed} if
whenever $z^\alpha \in S$
and $z_j | z^\alpha$, then for all $\ell < j$,
the monomials $z^\alpha\frac{1}{z_j}$ and
$z^\alpha\frac{z_\ell}{z_j}$ are in $S$.
If $X \subset \sO(U)$ is a monomial subspace, then we say that $X$ is
affine-Borel-fixed if the basis of monomials that generates $X$ is
affine-Borel-fixed.

Take an invertible affine self map $\tau$ of $\C^n$ such that $\tau(0) \in U$.
For a subspace $X \subset \sO(U)$, define a subspace $X \circ \tau \subset \sO(U')$
where $U' = \tau^{-1}(U)$.  Notice that $0 \in U'$ and we may again
talk about initial monomials as above.

The initial monomial space $\operatorname{in}(X)$ is not always preserved
under such precomposition
with affine maps.
For example, let
$X$ be the span of
$\{ 1, z_2 \}$, so $X = \operatorname{in}(X)$.
For a generic choice of an affine map $\tau$ we have
$\operatorname{in}(X \circ \tau) = \operatorname{span} \{ 1, z_1 \}$.
Notice that $\{ 1, z_2 \}$ is not affine-Borel-fixed,
but $\operatorname{in}(X \circ \tau)$ is.

We need to prove an analogue to Galligo's Theorem (see Theorem 1.27 in
\cite{Green:gin}).  The finite dimensional version of the result is proved in
\cite{GLV}.

The set $\operatorname{Aff}(n)$ of affine self maps of $\C^n$ can be
parametrized as $M_n(\C) \times \C^n$, or $\C^{n^2} \times \C^n$.  We use
the standard topology on this set.

\begin{thm} \label{thm:ourgalligo}
Suppose $U \subset \C^n$ and $0 \in U$, and
let $X \subset \sO(U)$ be a vector subspace.
There is some neighbourhood $\sN$ of the identity in $\operatorname{Aff}(n)$
and $\sA \subset \sN$ of second category
(countable intersection of open dense subsets of $\sN$), such that
for $\tau \in \sA$,
the space
$\operatorname{in}(X \circ \tau)$ is affine-Borel-fixed.
Furthermore, $\operatorname{in}(X \circ \tau) = \operatorname{in}(X \circ
\tau')$ for any two affine $\tau$ and $\tau'$ in $\sA$.

Finally, if $X$ is finite dimensional, $\sA$ is open and
dense in $\sN$.
\end{thm}

So after a generic affine transformation the initial
monomial subspace is always
the same space.
It will also not be necessary to assume $0 \in U$. 
We pick a generic affine $\tau$ such that $\tau(0) \in
U$.  Then $0 \in \tau^{-1}(U)$, and so it makes sense to take
$\operatorname{in}(X \circ \tau)$ for $X \subset \sO(U)$.
Before we prove the theorem, let us make the following definition.

\begin{defn}
Let $U \subset \C^n$ and let $X \subset \sO(U)$ be a subspace.
Take a generic affine $\tau$ such that $\tau(0)
\in U$.  Define the \emph{generic initial monomial subspace}
\begin{equation}
\operatorname{gin}(X) = \operatorname{in}(X \circ \tau) .
\end{equation}
\end{defn}

\begin{remark}
While the proof of the theorem may seem formal, we are
using affine transformations and precomposing formal power series with
affine transformations may not make sense.
\end{remark}

\begin{proof}[Proof of Theorem \ref{thm:ourgalligo}]
The finite dimensional version of the theorem is proved in
\cite{GLV} as Proposition 1 (or Proposition 5.5 on arXiv).

Suppose $X$ is not finite dimensional.
The space $\sO(U)$ is a separable Frech\'et space.  For
the subspace $X \subset \sO(U)$, pick a countable set
$\{ f_k \}_{k=1}^\infty$ in $X$ such that if
\begin{equation}
X_m = \operatorname{span} \{ f_1, f_2, \ldots, f_m \} ,
\end{equation}
then 
\begin{equation}
\bigcup_m X_m \subset X \subset
\overline{X} = \overline{\bigcup_m X_m} .
\end{equation}

Apply the finite dimensional result to $X_m$ for each $m$.  The $\tau$
varies over countably many open dense sets and the
intersection of those sets is a second category set as claimed.

Given any $k$, the initial $k$ monomials
in $\{ f_1, f_2, \ldots, f_m \}$ must stabilize as $m$ grows.  That is,
the first $k$ monomials are the same for all large enough $m$.
We therefore obtain a sequence of initial monomials.  To obtain the first
$k$ monomials, simply go far enough
until the initial $k$ monomials stabilize.
Then $z^\alpha$ is in this sequence if and only if
$z^\alpha \in \operatorname{in}\bigl(\bigcup_m X_m\bigr)$.

It is left to show that 
\begin{equation}
\operatorname{in}\bigl(\overline{X}\bigr)
=
\operatorname{in}\biggl(\overline{\bigcup_m X_m}\biggr)
=
\operatorname{in}\biggl(\bigcup_m X_m\biggr) .
\end{equation}
Suppose this is not true.
Find the first (maximal) monomial $z^\alpha$ such that
$z^\alpha \in \operatorname{in}\bigl(\overline{X}\bigr)$,
but
$z^\alpha \not\in \operatorname{in}\bigl(\bigcup_m X_m\bigr)$.
Suppose $z^\alpha$ is the $k$th monomial in the ordering given.
Let $\pi_k$ be the projection of $\sO(U)$ onto the space spanned by
the first $k$ monomials.  By Cauchy's formula, $\pi_k$ is continous.
The dimension of $\pi_k \bigl( \bigcup_{m=1}^M X_m  \bigr)$ stabilizes
as $M$ grows, and hence for large enough $M$
\begin{equation}
\pi_k \bigl( X_M  \bigr)
=
\pi_k \biggl( \bigcup_{m=1}^\infty X_m  \biggr)
=
\pi_k \bigl( \overline{X}  \bigr) .
\end{equation}
There exists an $f \in \overline{X}$ with $z^\alpha = \operatorname{in}(f)$.
By the above equality there must exist
a $g \in X_M$ such that $z^\alpha = \operatorname{in}(g)$.  That is a
contradiction.
\end{proof}

We now generalize the gins to biholomorphic maps.
Let $X \subset \sO(U)$ be a subspace and
let $f \colon U' \to U$ be a holomorphic map.  Define
$X \circ f$ to be the vector subspace of $\sO(U')$ consisting of
all $F \circ f$ for $F \in X$.

\begin{lemma} \label{lemma:biholgin}
Let $X$ be a vector subspace of $\sO(U)$ and let $f \colon U' \to U$ be a
biholomorphic map.  Then $\operatorname{gin}(X) = \operatorname{gin}(X \circ
f)$.
\end{lemma}

\begin{proof}
Suppose $0$ is in both $U'$ and $U$, and suppose $f(0) = 0$.
Composing $f$ with an invertible linear map does not change the gin (we must
change $U$ appropriately).  Hence, assume 
$f'(0)$ is the identity.
Let
\begin{equation}
f(z) = z + E(z) ,
\end{equation}
where $E$ is of order 2 and higher.  We only need to show that
$\operatorname{in}(X) = \operatorname{in}(X \circ f)$.

Suppose $z^\alpha \in \operatorname{in}(X)$, that is, there exists an
element of $X$ of the form
\begin{equation}
g(z) = z^\alpha  + \sum_{\beta < \alpha} c_\beta z^\beta .
\end{equation}
In
$(g \circ f)$, the terms from $E$ create
monomials of degree strictly higher than $\abs{\alpha}$.  So
\begin{equation}
(g \circ f)(z) = z^\alpha  + \sum_{\beta < \alpha} d_\beta z^\beta .
\end{equation}
for some $d_\beta$.  Therefore $z^\alpha \in \operatorname{in}(X \circ f)$.

Since $f^{-1}(z) = z + F(z)$ for some $F$ or higher order we obtain by
symmetry that if $z^\alpha \in \operatorname{in}(X \circ f)$ then
if $z^\alpha \in
\operatorname{in}(X \circ f \circ f^{-1}) = 
\operatorname{in}(X)$.
\end{proof}

\begin{lemma} \label{lemma:multgin}
Let $X$ be a vector subspace of $\sO(U)$, $0 \in U$,
and let $\varphi \colon U \to \C$ be a
holomorphic function that is not identically zero.  Then
$\operatorname{gin}(X) = \operatorname{gin}(\varphi  X)$.
\end{lemma}

\begin{proof}
As we are talking about gins, we precompose with a generic affine map,
which also precomposes $\varphi$, and therefore we assume that
$\varphi(0) \not= 0$.  In fact, without loss of generality we assume
$\varphi(0) = 1$.
Suppose $z^\alpha \in \operatorname{in}(X)$, that is, there is a $g \in
X$ of the form
\begin{equation}
g(z) = z^\alpha  + \sum_{\beta < \alpha} c_\beta z^\beta .
\end{equation}
Then
\begin{equation}
\varphi(z) g(z) = z^\alpha  + \sum_{\beta < \alpha} d_\beta z^\beta ,
\end{equation}
for some $d_\beta$ by multiplicativity of the monomial ordering.  Therefore
$z^\alpha \in \operatorname{in}(\varphi X)$.  By symmetry, 
$\operatorname{in}(X) = \operatorname{in}(\varphi X)$.
\end{proof}

%%%%%%%%%%%%%%%%%%%%%%%%%%%%%%%%%%%%%%%%%%%%%%%%%%%%%%%%%%%%%%%%%%%%%%%%%%%%

\section{Gins as invariants of maps} \label{section:ginsasinvar}

If $X \subset \sO(U)$ is a subspace and $U' \subset U$ is an open
set then the restriction $X|_{U'}$ (the space of restrictions
of maps from $X$ to $U'$) clearly has the same gin as $X$.
Let us extend gins to complex manifolds.  Let
$X \subset \sO(U)$ be a subspace where
$U$ is a connected complex manifold of dimension $n$.
As gin is invariant under biholomorphic 
transformations, it is well-defined on every chart for $U$.
If two (connected) charts overlap, then the gin
must be the same (compute the gin on the
intersection).  Therefore, $\operatorname{gin}(X)$ is well-defined.

\begin{thm} \label{thm:invmap}
Let $U, W$ be connected complex manifolds of dimension $n$.
Suppose $F \colon U \to \C^{N}$ and
$G \colon W \to \C^{N}$ are equivalent
in the sense that
there exists a biholomorphic map $\tau \colon W \to U$
and a linear fractional automorphism $\chi \in \operatorname{Aut}(\bP^N)$ such that
\begin{equation}
F \circ \tau = \chi \circ G .
\end{equation}
Then
\begin{equation}
\operatorname{gin}(\afspan{F}) =
\operatorname{gin}(\afspan{G}) .
\end{equation}
\end{thm}

\begin{proof}
As noted above, we take two charts of $U$ and $W$, 
and therefore, without loss of generality
we assume that $U$ and $W$ are domains in $\C^n$.

As $F \circ \tau$ and $G$ are equivalent via a target automorphism,
Lemma~\ref{lemma:targetspan} says
\begin{equation}
\afspan{(F \circ \tau)}
=
\varphi \afspan{G}.
\end{equation}
As 
$\afspan{(F \circ \tau)} = 
( \afspan{F} ) \circ \tau$, 
Lemma~\ref{lemma:biholgin} says
\begin{equation}
\operatorname{gin}\bigl(\afspan{(F \circ \tau)} \bigr)
=
\operatorname{gin}(\afspan{F} ) .
\end{equation}
Via Lemma~\ref{lemma:multgin}, we get 
\begin{equation}
\operatorname{gin}(\varphi \afspan{G} )
=
\operatorname{gin}(\afspan{G} ) .
\end{equation}
The result follows.
\end{proof}

The same proof is used for the following CR version.  In the CR version
we start with real-analytic CR maps. Since real-analytic CR maps extend to holomorphic
maps, we therefore work with the
extended holomorphic maps when taking gins and affine spans.

\begin{cor}
Let $M, M' \subset \C^n$ be connected real-analytic CR submanifolds.
Suppose $F \colon M \to Q(a,b)$ and
$G \colon M' \to Q(a,b)$ are real-analytic CR maps
equivalent in the sense that
there exists a real-analytic CR isomorphism $\tau \colon M' \to M$
and a linear fractional automorphism $\chi$ of $Q(a,b)$ such that
\begin{equation}
F \circ \tau = \chi \circ G .
\end{equation}
Then
\begin{equation}
\operatorname{gin}(\afspan{F}) =
\operatorname{gin}(\afspan{G}) .
\end{equation}
\end{cor}

Again we may take $M$ and $M'$ to be submanifolds of
a complex manifold of dimension $n$ instead of $\C^n$.

For examples we need only look at rational maps.  If the
map is rational we homogenize with $Z_0$ as before.
The generic initial monomial vector space is
simply the lowest degree part of the generic initial ideal
generated by the components of the homogenized map.  That is because the
lowest degree part of the ideal is the linear span of the components of the
map.  We must be careful that the ordering is compatible.
In particular, the ordering must respect the ``grading'' above, that is
in each degree if we set $Z_0=1$ we must still get a multiplicative
monomial ordering.  For example the standard reverse
lex ordering in each degree will not work, although the lex order will.

As an example take the Faran map $F(z) = (z_1^3,\sqrt{3}z_1z_2,z_2^3)$.  We add 1
and homogenize to get $(Z_0^3,Z_1^3,\sqrt{3}Z_0Z_1Z_2,Z_2^3)$.
We compute the gin in the graded lex order.
The degree 3 part of the gin is generated by
\begin{equation}
Z_0^3, Z_0^2 Z_1, Z_0^2 Z_2, Z_0Z_1^2 .
\end{equation}
Therefore, 
\begin{equation}
\operatorname{gin}(\afspan{F}) =
\operatorname{span} \{
1, z_1, z_2, z_1^2
\} .
\end{equation}

%%%%%%%%%%%%%%%%%%%%%%%%%%%%%%%%%%%%%%%%%%%%%%%%%%%%%%%%%%%%%%%%%%%%%%%%%%%%

\section{Invariants of the quotient of a hyperquadric map} \label{section:quotient}

When the source is either the ball $\ballsig{a}{b}$ in the holomorphic case
or $Q(a,b)$ in the CR case, we obtain further invariants by considering
the quotient of the function composed with the defining function of the
target divided by the defining function of the source as we did in the
rational case.

Recall the definition 
$\norm{z}_b^2 =
-\sum_{j=1}^{b} \abs{z_j}^2 +
\sum_{j=b+1}^{a+b} \abs{z_j}^2$.
The defining function for $Q(a,b)$ is then $\norm{z}_b^2 - 1 = 0$.
Suppose $F \colon U \subset Q(a,b) \to Q(A,B)$ is a real-analytic CR map.
A real-analytic CR map is
a holomorphic map of a neighborhood of $U$, and so we will identify
$F$ with this holomorphic map.
Define the quotient $q$ via
\begin{equation}
\norm{F(z)}_B^2-1 = \bigl(\norm{z}_b^2-1\bigr)q(z,\bar{z}) .
\end{equation}
Near some point find the holomorphic decomposition of $q$:
\begin{equation}
q(z,\bar{z}) = \norm{h_+(z)}^2-\norm{h_-(z)}^2 ,
\end{equation}
where $h_+$ and $h_-$ are possibly $\ell^2$ valued holomorphic maps.
Write
\begin{equation}
H(q) = \{ h_+,h_- \}
\end{equation}
for the set of functions in the holomorphic decomposition of $q$.
The holomorphic decomposition is not unique and depends on the point.
However we do have the following lemma.

\begin{lemma} \label{lemma:holdecompinv}
Suppose $q \colon U \subset \C^n \to \R$ is real-analytic and $U$ is a
connected open set.  If $H_1(q)$ and $H_2(q)$ are two holomorphic decompositions
at two points of $U$, then
\begin{equation}
\operatorname{gin}\bigl(\operatorname{span} H_1(q)\bigr)
=
\operatorname{gin}\bigl(\operatorname{span} H_2(q)\bigr) .
\end{equation}
\end{lemma}

\begin{proof}
By a standard connectedness argument using a path between the two points we
only need to consider the case where the domains of convergence of $H_1$ and
$H_2$ overlap.  Assume that we work on this
intersection.  So we only need to show that if 
\begin{equation} \label{eq:decomsame}
\norm{h_+(z)}^2-\norm{h_-(z)}^2 = \norm{h'_+(z)}^2-\norm{h'_-(z)}^2
\end{equation}
for maps $h_+$, $h_-$, $h'_+$, $h'_-$ converging on a fixed open set,
with $\{h_+,h_-\}$ and $\{h'_+,h'_-\}$ linearly independent sets,
then
$\operatorname{gin}\bigl(\operatorname{span}\{h_+,h_-\}\bigr)
=
\operatorname{gin}\bigl(\operatorname{span}\{h'_+,h'_-\}\bigr)$.
From \eqref{eq:decomsame} we have 
$\operatorname{span}\{h_+,h_-\} = \operatorname{span}\{h'_+,h'_-\}$, and so
the result follows.
\end{proof}

Thus, we do not need to specify which point and which decomposition is used.

\begin{thm}
Suppose $F \colon U \subset Q(a,b) \to Q(A,B)$ and
$G \colon V \subset  Q(a,b) \to Q(A,B)$ are real-analytic CR maps 
equivalent in the sense that
there exists a linear fractional automorphism $\tau$ of $Q(a,b)$ 
and a linear fractional automorphism $\chi$ of $Q(A,B)$ such that
\begin{equation}
F \circ \tau = \chi \circ G .
\end{equation}
Define the quotients $q_F$ and $q_G$ via
\begin{equation}
\norm{F(z)}_B^2-1 = \bigl(\norm{z}_b^2-1\bigr)q_F(z,\bar{z})
\qquad \text{and}\qquad
\norm{G(z)}_B^2-1 = \bigl(\norm{z}_b^2-1\bigr)q_G(z,\bar{z}) .
\end{equation}
Then taking holomorphic decompositions of $q_F$ and $q_G$ at any point
in $U$ and $V$ respectively 
\begin{equation}
\operatorname{gin}\bigl(\operatorname{span} H(q_F)\bigr) =
\operatorname{gin}\bigl(\operatorname{span} H(q_G)\bigr) .
\end{equation}
\end{thm}

\begin{proof}
Write the linear fractional maps $\tau$ and $\chi$ as
\begin{equation}
\tau = \frac{\tau'}{\tau''}
\qquad \text{and} \qquad
\chi = \frac{\chi'}{\chi''} .
\end{equation}
for affine maps $\tau'$, $\tau''$, $\chi'$, and $\chi''$.  Then
\begin{equation}
\norm{\tau'(z)}_b^2-\abs{\tau''(z)}^2 = \pm
(\norm{z}_b^2-1)
\end{equation}
so
\begin{equation}
\abs{\tau''(z)}^2
(\norm{\tau(z)}_b^2-1) = \pm
(\norm{z}_b^2-1) .
\end{equation}
The $\pm$ is there again in case $\tau$ switches the sides of $Q(a,b)$.
We have the same equation for $w$ and $\chi$ with $A,B$ instead of $a,b$.

We have $\norm{F(\tau(z))}_{B}^2+1 = \norm{\chi \circ G(z)}_{B}^2+1$.
Then
\begin{equation}
\begin{split}
(\norm{z}_{b}^2-1) q_F\bigl(\tau(z),\overline{\tau(z)}\bigr)
& =
\pm
\abs{\tau''(z)}^2
\bigl(\norm{\tau(z)}_{b}^2-1\bigr) q_F\bigl(\tau(z),\overline{\tau(z)}\bigr)
\\
& =
\pm
\abs{\tau''(z)}^2
\bigl(\norm{F(\tau(z))}_{B}^2-1\bigr)
\\
& =
\pm
\abs{\tau''(z)}^2
\bigl(\norm{\chi \circ G(z)}_{B}^2-1\bigr)
\\
& =
\pm
\abs{\frac{\tau''(z)}{\chi''\bigl(G(z)\bigr)}}^2
\bigl(\norm{G(z)}_{B}^2-1\bigr)
\\
& =
\pm
\abs{\frac{\tau''(z)}{\chi''\bigl(G(z)\bigr)}}^2
\bigl(\norm{z}_b^2-1\bigr)q_G(z,\bar{z}) .
\end{split}
\end{equation}
In other words
\begin{equation}
q_F\bigl(\tau(z),\overline{\tau(z)}\bigr) = \pm
\abs{\frac{\tau''(z)}{\chi''\bigl(G(z)\bigr)}}^2
q_G(z,\bar{z}) .
\end{equation}
Multiplying by absolute value squared of a holomorphic function
multiplies the elements of the holomorphic decomposition by
that function and so does not change the gin by Lemma~\ref{lemma:multgin}.
Similarly composing with $\tau$ does not change the gin either by
Lemma~\ref{lemma:biholgin}.
\end{proof}

There is an equivalent theorem for holomorphic maps, when the automorhphism
maps on the source are simply the linear fractional automorphisms of
$\ballsig{a}{b}$, that is, automorphisms of $\bP^{a+b}$ preserving the
closure of $\ballsig{a}{b}$ in $\bP^{a+b}$.  We lose no generality
if we also allow swapping sides (when $a=b+1$) and hence we consider all
linear fractional automorphisms of $Q(a,b)$, that is, self maps of
$\bP^{a+b}$ preserving $HQ(a,b+1)$.

\begin{thm}
Suppose $F \colon U \subset \C^{a+b} \to \C^{A+B}$ and
$G \colon V \subset \C^{a+b} \to \C^{A+B}$ are holomorphic maps 
equivalent in the sense that
there exists a linear fractional automorphism $\tau$ of $Q(a,b)$,
where $\tau(V) = U$,
and a linear fractional automorphism $\chi$ of $Q(A,B)$ such that
\begin{equation}
F \circ \tau = \chi \circ G .
\end{equation}
Define the quotients $q_F$ and $q_G$ as before:
\begin{equation}
\norm{F(z)}_B^2-1 = \bigl(\norm{z}_b^2-1\bigr)q_F(z,\bar{z})
\qquad \text{and}\qquad
\norm{G(z)}_B^2-1 = \bigl(\norm{z}_b^2-1\bigr)q_G(z,\bar{z}) .
\end{equation}
Then taking holomorphic decompositions of $q_F$ and $q_G$ at any point
in $U$ and $V$ respectively 
\begin{equation}
\operatorname{gin}\bigl(\operatorname{span} H(q_F)\bigr) =
\operatorname{gin}\bigl(\operatorname{span} H(q_G)\bigr) .
\end{equation}
\end{thm}

%%%%%%%%%%%%%%%%%%%%%%%%%%%%%%%%%%%%%%%%%%%%%%%%%%%%%%%%%%%%%%%%%%%%%%%%%%%%

\section{Gins of real-analytic functions}
\label{section:holdecomp}

We end the article with a remark that Lemma~\ref{lemma:holdecompinv}
may be of independent interest, not only for mapping
questions.  Let us consider the gin
of a holomorphic decomposition of real-analytic functions.
That is, for a domain $U \subset \C^n$,
consider two real-analytic functions
of real-analytic functions $r_1 \colon U \to \R$ and $r_2 \colon U \to \R$.
Then say the functions are biholomorphically equivalent
if there exists a biholomorphism $F \colon U \to U$,
such that $r_1 = r_2 \circ F$.

If we take $H(r_j)$ to be the holomorphic decomposition of $r_j$ at any
point of $U$, then
Lemma~\ref{lemma:holdecompinv} says that if $r_1$ and $r_2$ are equivalent
as above, then
\begin{equation}
\operatorname{gin}\bigl(\afspan{H(r_1)}\bigr)
=
\operatorname{gin}\bigl(\afspan{H(r_2)}\bigr) .
\end{equation}

The gin as defined above is not a pointwise invariant.  That is, a
real-analytic function defined on a connected open set has the same gin near every
point.  By the same argument as before we can also define the gin of a
real-analytic function on a connected complex manifold $U$ by noting
that the gin is already well-defined if we take any connected chart.

%%%%%%%%%%%%%%%%%%%%%%%%%%%%%%%%%%%%%%%%%%%%%%%%%%%%%%%%%%%%%%%%%%%%%%%%%%%%

\def\MR#1{\relax\ifhmode\unskip\spacefactor3000 \space\fi%
  \href{http://www.ams.org/mathscinet-getitem?mr=#1}{MR#1}}

\end{document}